\documentclass[12pt]{article}

\usepackage{amsmath, amsthm, amssymb}
\usepackage{graphicx}
\usepackage{microtype}
\usepackage{geometry}
\usepackage[colorlinks=true, linkcolor=blue, citecolor=blue, urlcolor=blue]{hyperref}

\newtheorem{theorem}{Theorem}[section]
\newtheorem{lemma}[theorem]{Lemma}

\DeclareMathOperator{\Mod}{Mod}

\title{Generating the mapping class group of a nonorientable
       surface of genus $g \geq 13$ by two elements}

\author{
  Berkay Aybak\thanks{Department of Mathematics,
    Middle East Technical University, Ankara, Turkey.}
  \and
  Hasan Ozden\thanks{Department of Mathematics,
    Mustafa Kemal University, Hatay, Turkey.}
}

\date{\today}
\begin{document}

\maketitle

\begin{abstract}
Let $N_g$ be a closed, connected, nonorientable surface of genus $g$.
We prove that for $g \ge 13$, the mapping class group $\Mod(N_g)$ can be
generated by exactly two elements.
This improves the previously known bound of $g \ge 19$.
\end{abstract}



\bigskip

%%------------------------------------------------------------
\section{Introduction}
%%------------------------------------------------------------

Let $N_g$ be a closed, connected, nonorientable surface of genus $g$.
The \emph{mapping class group} $\Mod(N_g)$ is the group of isotopy classes
of self-diffeomorphisms of $N_g$.
Finding minimal generating sets for mapping class groups is a long-standing
problem in geometric group theory and low-dimensional topology.

Because $\Mod(N_g)$ is non-abelian, every generating set must contain at
least two elements.
Szepietowski \cite{Szepietowski} proved that $\Mod(N_g)$ is generated by
three elements for all $g \ge 3$.
Later, Altun\"{o}z, Pamuk, and Y\i ld\i z \cite{APY} answered a question
posed in \cite[Problem~3.1(a)]{Farb} by showing that $\Mod(N_g)$ can be
generated by exactly two elements for $g \ge 19$, one of which is a periodic
rotation of order $g$.

The genus bound $g \ge 19$ in \cite{APY} is essentially a geometric
limitation: to isolate the necessary Dehn twists via the lantern relation,
the surface must be large enough so that the supports of the shifted
generators do not accidentally overlap.
In this note we lower the genus bound by choosing a different second
generator and employing a more efficient sequence of commutators.

Our main results are the following.

\begin{theorem}\label{thm:main}
  For $g \ge 14$, the mapping class group $\Mod(N_g)$ is generated by two
  elements, $T$ and $u_{g-4}A_2 C_2^{-1}$.
\end{theorem}

\begin{theorem}\label{thm:main2}
  The mapping class group $\Mod(N_{13})$ is generated by two
  elements, $T$ and $u_9A_2 B_1^{-1}$.
\end{theorem}

%%------------------------------------------------------------
\section{Preliminaries and the main results}
%%------------------------------------------------------------

Throughout this paper, curves on $N_g$ and self-diffeomorphisms of $N_g$
are considered up to isotopy.
We denote the right-handed Dehn twist about a two-sided simple closed curve
$a$ by the corresponding capital letter $A$.
We use the standard model for $N_g$ as a sphere with $g$ crosscaps
$\mathcal{C}_1,\dots,\mathcal{C}_g$ arranged in a circular position,
and we let $T$ denote the rotation by $2\pi/g$ that maps $\mathcal{C}_i$
to $\mathcal{C}_{i+1}$ (indices taken modulo $g$).
For a model of the surface $N_g$ and the classical two-sided curves used in
what follows, we refer to \autoref{fig:model}.
The \emph{geometric intersection number} of two isotopy classes of curves $a$
and $b$ is
\[
  i(a,b) = \min\bigl\{|\alpha\cap\beta| : \alpha\in a,\,\beta\in b\bigr\}.
\]

For a family of two-sided simple closed curves, the union of their regular
neighborhoods is an orientable subsurface, which allows us to apply Dehn-twist
relations from the orientable setting; for these we always refer to
\cite{FarbMargalit}.

\begin{lemma}\label{commutativity}
  Let $a$ and $b$ be isotopy classes of two-sided simple closed curves.
  The following statements are equivalent\/\textup{:}
  \begin{enumerate}
    \item[\textup{(i)}]  $i(a,b)=0$,
    \item[\textup{(ii)}] $AB = BA$,
    \item[\textup{(iii)}] $A(b) = b$.
  \end{enumerate}
\end{lemma}

\begin{lemma}\label{prop3.12}
  Let $a$ and $b$ be isotopy classes of two-sided simple closed curves with
  $i(a,b)=1$.
  Then
  \[
    AB(a) = b.
  \]
\end{lemma}

For nonorientable surfaces, Dehn twists alone do not suffice to generate
$\Mod(N_g)$; additional elements are needed.
Let $K$ be a subsurface of $N_g$ (for $g \geq 2$) homeomorphic to a Klein
bottle with one boundary component; $K$ contains exactly two crosscaps.
A \emph{crosscap transposition} is a self-homeomorphism of $N_g$ supported
on such a subsurface $K$ that interchanges its two crosscaps.
The effect of a crosscap transposition on an arc $c$ is illustrated in
\autoref{fig:crosscaptransposition}.
We write $u_{\mu,\alpha}$ or $u_i$ for a crosscap transposition depending on
the context; the notation $u_i$ stands for the transposition of the
$i$-th and $(i{+}1)$-th crosscaps. By $\alpha_i$, we denote the boundary for the support of the crosscap transposition $u_i$.

\begin{figure}[htbp]
  \centering
  \includegraphics[width=0.7\linewidth]{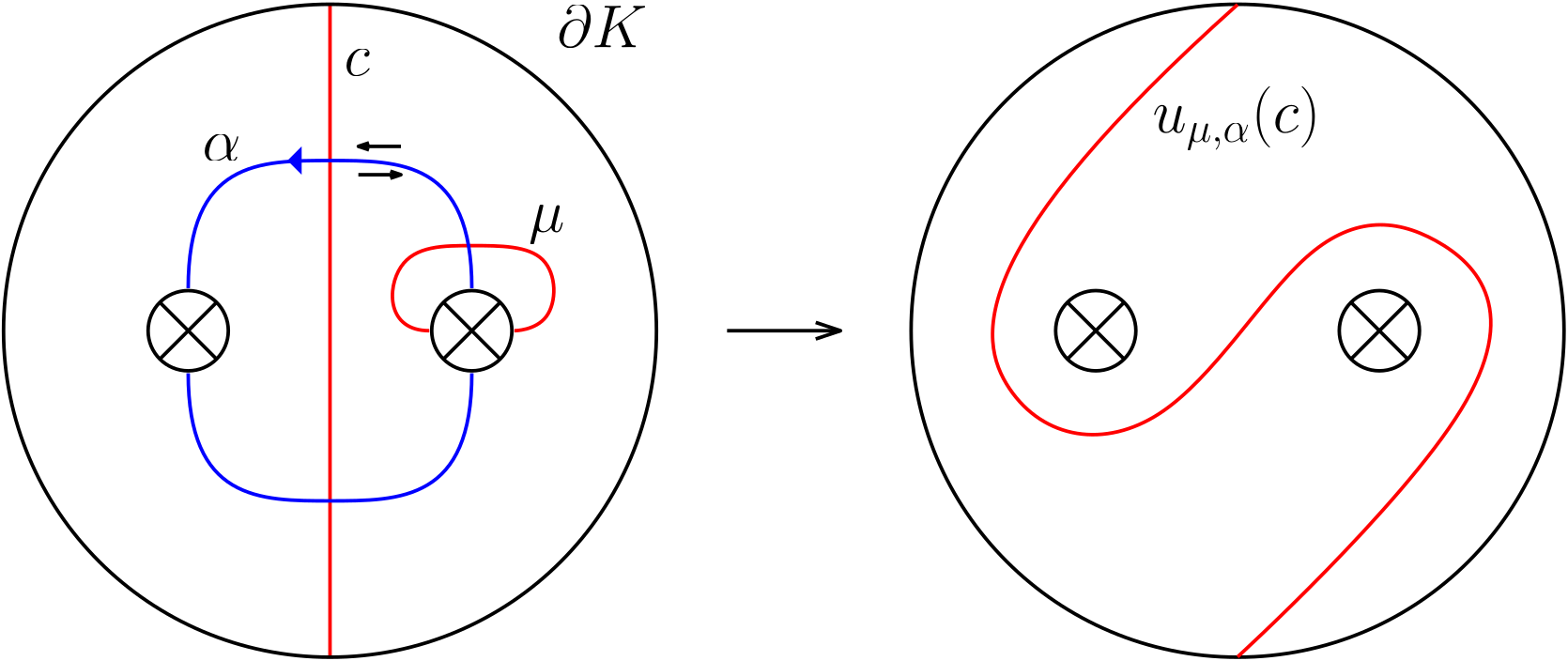}
  \caption{A crosscap transposition and the image of the arc $c$.}
  \label{fig:crosscaptransposition}
\end{figure}

The following standard conjugation relations will be used throughout; see
\cite[Lemma~0]{Chillingworth} and \cite[Relation~2.3]{LesniakSzepietowski}.

\begin{itemize}
  \item If $f(a) = b$, then $fAf^{-1} = B^s$, where $s = \pm 1$ according as
        $f$ preserves or reverses the local orientation.
  \item $f\,u_{\mu,\alpha}\,f^{-1} = u_{f(\mu),f(\alpha)}$.
\end{itemize}

Since the rotation $T$ maps $\mathcal{C}_i$ to $\mathcal{C}_{i+1}$, we have
in particular $Tu_i T^{-1} = u_{i+1}$.
Also, note that $T^{k-1}(a_2)=\gamma_{k}$, where $\gamma_{k}$ is a two-sided curve passing through $k$, $(k+1)$, $(k+2)$ and $(k+3)$ indexed crosscaps (modulo $g$).
\begin{figure}[htbp]
  \centering
  \includegraphics[width=1.0\linewidth]{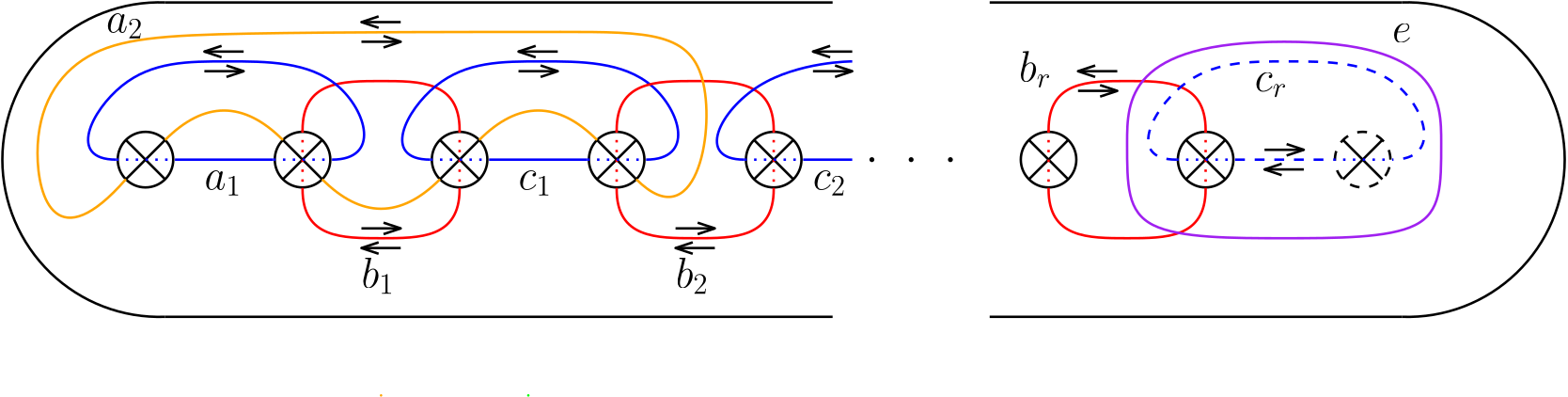}
  \caption{A model of $N_g$ for $g = 2r+2$ (or $g = 2r+1$ after removing
    the last crosscap; in that case the curve $c_r$ disappears) and some
    classical two-sided curves.}
  \label{fig:model}
\end{figure}

Our proof reduces to establishing the following generating set.

\begin{theorem}[\cite{APY}]\label{thm:base}
  For $g \ge 7$, the mapping class group $\Mod(N_g)$ is generated by
  $T$, $A_1A_2^{-1}$, $B_1B_2^{-1}$, and $u_{g-1}$.
\end{theorem}

We can now prove \autoref{thm:main}.

\begin{proof}[Proof of \autoref{thm:main}]
Let $G$ denote the subgroup of $\Mod(N_g)$ generated by $T$ and
$u_{g-4}A_2C_2^{-1}$.
Our goal is to show that $G$ contains $A_1A_2^{-1}$, $B_1B_2^{-1}$,
and $u_{g-1}$ from \autoref{thm:base}, and hence $G = \Mod(N_g)$.

Set $G_1 = u_{g-4}A_2C_2^{-1}$.
Conjugating by $T^3$ gives
\[
  G_2 = T^3 G_1 T^{-3} = u_{g-1}\Gamma_4 B_4^{-1} \in G.
\]
Using $G_1$ and $G_2$ we form the commutator
\[
  G_3 = (G_2 G_1)\,G_2\,(G_2 G_1)^{-1} = u_{g-1} A_2 B_4^{-1} \in G,
\]
and then
\[
  G_3 G_2^{-1} = A_2\Gamma_4^{-1} \in G.
\]
Conjugating by $T$ and then by $T^3$ successively yields
\begin{align*}
  T(A_2\Gamma_4^{-1})T^{-1}   &= \Gamma_2\Gamma_5^{-1} \in G, \\
  T^3(\Gamma_2\Gamma_5^{-1})T^{-3} &= \Gamma_5\Gamma_8^{-1} \in G.
\end{align*}
Forming the product
\[
  G_4 = (\Gamma_2\Gamma_5^{-1})(\Gamma_5\Gamma_8^{-1}) = \Gamma_2\Gamma_8^{-1} \in G.
\]
Since $\gamma_2$ intersects $a_2$ (with $i(\gamma_2,a_2)=1$), by
\autoref{prop3.12} and the conjugation formula,
\[
  G_5 = (G_4 G_3)\,G_4\,(G_4 G_3)^{-1} = A_2\Gamma_8^{-1} \in G.
\]
Hence $G_4 G_5^{-1} = A_2\Gamma_2^{-1} \in G$.
Conjugating this element with the powers of $T$ gives elements of
the form $\Gamma_i\Gamma_{i+1}^{-1} \in G$ for $i = 2,\ldots,g-4$.
Telescoping these yields
\[
  G_6 = (A_2\Gamma_2^{-1})(\Gamma_2\Gamma_3^{-1})\cdots(\Gamma_5\Gamma_6^{-1})
       = A_2\Gamma_6^{-1} \in G.
\]
Conjugating $G_6$ by $G_6 G_1$ gives
\[
  G_7 = (G_6 G_1)\,G_6\,(G_6 G_1)^{-1} = A_2 C_2^{-1} \in G.
\]
We can then isolate a single crosscap transposition:
\[
  G_1 G_7^{-1} = (u_{g-4} A_2 C_2^{-1})(C_2 A_2^{-1}) = u_{g-4} \in G.
\]
Conjugating $u_{g-4}$ with powers of $T$ yields $u_j \in G$ for all
$j = 1,\ldots,g-1$.
In particular $u_{g-1} \in G$, so from $G_3 = u_{g-1}A_2B_4^{-1} \in G$
we extract $A_2 B_4^{-1} \in G$.

With $G_7 = A_2 C_2^{-1} \in G$ and $A_2 B_4^{-1} \in G$ we form
\[
  G_7^{-1}(A_2 B_4^{-1}) = C_2 B_4^{-1} \in G.
\]
Conjugating by powers of $T$:
\begin{align*}
  G_8 &= T^{-3}(C_2 B_4^{-1})T^3 = B_1 C_2^{-1} \in G, \\
  G_9 &= T^2(B_1 C_2^{-1})T^{-2} = B_2 C_3^{-1} \in G.
\end{align*}
The commutator
\[
  (G_8 G_9^{-1})\,G_8\,(G_8 G_9^{-1})^{-1} = B_1 B_2^{-1} \in G
\]
gives the third generator of \autoref{thm:base}.

Finally, conjugating $B_1 B_2^{-1}$ with powers of $T$ shows that
$A_1 C_1^{-1}$ and $C_1 C_2^{-1}$ belong to $G$.
Their product gives $A_1 C_2^{-1} \in G$, and then
\[
  (A_1 C_2^{-1})(C_2 A_2^{-1}) = A_1 A_2^{-1} \in G
\]
supplies the last generator of \autoref{thm:base}, completing the proof.
\end{proof}

\begin{figure}[htbp]
  \centering
  \includegraphics[width=0.6\linewidth]{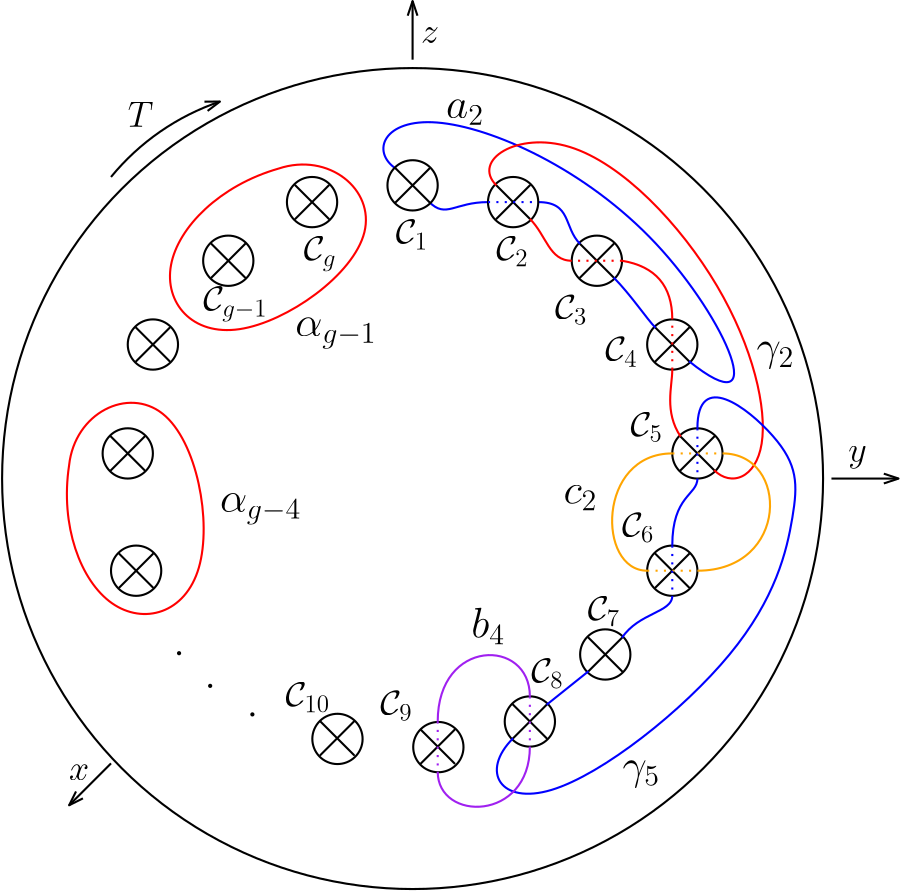}
  \caption{The rotation $T$ and some of the curves appearing in the proof of \autoref{thm:main}.}
  \label{fig:rotation}
\end{figure}

 The ideas concerning the proof for \autoref{thm:main2} are quite similar, with a subtle change; we use a different method to obtain the element $A_2 \Gamma^{-1}_2$. 
 \begin{proof}[Sketch of the Proof of \autoref{thm:main2}]
 Let $H$ denote the subgroup of $\Mod(N_{13})$ generated by $T$ and
$u_{9}A_2B_1^{-1}$.
We aim to show that $H$ contains $A_1A_2^{-1}$, $B_1B_2^{-1}$,
and $u_{12}$ from \autoref{thm:base}, and hence $H = \Mod(N_{13})$. Following the ideas of the proof of  \autoref{thm:main} closely, we obtain the elements: \begin{itemize}
    \item $H_1=u_{9}A_2B_1^{-1}$,
    \item $H_2 = T^3 H_1 T^{-3} = u_{12}\Gamma_4 C_2^{-1},$
    \item $H_3=(H_2H_1)H_2(H_2H_1)^{-1}=u_{12}A_2C_2^{-1}$,
\end{itemize}
Using these, we also get $H_3H_2^{-1}=A_2\Gamma_4^{-1}\in H.$ Conjugating this element by $T$ and then by $T^3$ successively yields
\begin{align*}
  T(A_2\Gamma_4^{-1})T^{-1}   &= \Gamma_2\Gamma_5^{-1} \in H, \\
  T^3(\Gamma_2\Gamma_5^{-1})T^{-3} &= \Gamma_5\Gamma_8^{-1} \in H.
\end{align*}
These elements form the product
\[
  H_4 = (\Gamma_2\Gamma_5^{-1})(\Gamma_5\Gamma_8^{-1}) = \Gamma_2\Gamma_8^{-1} \in H.
\]
By further conjugating $H_4$ with $T^3$, we obtain $$T^3H_4T^{-3}=\Gamma_5 \Gamma_{11}^{-1} \in H.$$ By using this element, we also form the product $$(\Gamma_2\Gamma_5^{-1})(\Gamma_5 \Gamma_{11}^{-1})=\Gamma_2\Gamma_{11}^{-1} \in H.$$ 
On the other hand, conjugating $H_3H_2^{-1}=A_2\Gamma_4^{-1}$ with $T^{-3}$ gives $$T^{-3}(A_2\Gamma_4^{-1})T^{3}=\Gamma_{11}A_2^{-1}\in H.$$ Also, product of gathered elements yields $$H_5=(\Gamma_2\Gamma_{11}^{-1})(\Gamma_{11}A_2^{-1})=\Gamma_2 A_2^{-1} \in H.$$ The rest of the proof follows as in \autoref{thm:main}, with slight adjustments; we gather the elements \begin{itemize}
    \item $H_6 = (A_2\Gamma_2^{-1})(\Gamma_2\Gamma_3^{-1})\cdots(\Gamma_5\Gamma_6^{-1})
       = A_2\Gamma_6^{-1}$,
    \item $H_7=(H_6H_3)H_6(H_6H_3)^{-1}=A_2C_2^{-1}$, 
    \item $H_3H^{-1}_7=u_{12},$
    \item $u_{9}^{-1}H_1=A_2B_1^{-1},$
    \item $H_8=(B_1A_2^{-1})H_7=B_1C_2^{-1}.$
       \end{itemize}
\end{proof}

\section*{Acknowledgements}
The authors would like to thank Altunöz, Pamuk, and Yıldız for many helpful discussions and valuable insights. 
%%------------------------------------------------------------

\end{document}